\documentclass[reqno,final,a4paper]{amsart}
\usepackage{color}
\definecolor{darkgray}{gray}{0.9}
\usepackage{amsmath, amssymb, amsthm, thmtools}
    \usepackage[colorlinks=true,
            linkcolor=darkgray,
            citecolor=darkgray,
            urlcolor=darkgray]{hyperref}
\usepackage{mathrsfs}
\usepackage{mathtools}
\usepackage[linesnumbered,ruled,vlined]{algorithm2e}

\SetCommentSty{mycommfont}
\SetKwInOut{Input}{input}\SetKwInOut{Output}{output}

\usepackage[noadjust]{cite}
\usepackage[noabbrev,capitalize,nameinlink]{cleveref}
\crefname{equation}{}{}
\usepackage{fullpage}
\usepackage{graphics}
\usepackage{tikzscale}
\usepackage{pifont}
\usepackage{tikz}
\usetikzlibrary{fit, backgrounds, shapes, calc}
\usepackage{pgfplots}
\pgfplotsset{compat=1.12}
\usepackage{bbm}
\usepackage[T1]{fontenc}
\usetikzlibrary{arrows.meta}

\usepackage{environ}
\usepackage{framed}
\usepackage{url}
\usepackage{thm-restate}

\usepackage[noend]{algpseudocode}
\usepackage[labelfont=bf]{caption}
\usepackage{framed}
\usepackage[framemethod=tikz]{mdframed}
\usepackage{appendix}
\usepackage{graphicx}
\usepackage[textsize=tiny]{todonotes}
\usepackage{tcolorbox}
\usepackage{xcolor}
\usepackage[shortlabels]{enumitem}
\crefformat{enumi}{#2#1#3}
\crefrangeformat{enumi}{#3#1#4 to~#5#2#6}
\crefmultiformat{enumi}{#2#1#3}
{ and~#2#1#3}{, #2#1#3}{ and~#2#1#3}
\allowdisplaybreaks

\usepackage{hyperref}
\usepackage{nameref}

\makeatletter
\newcommand{\namedlabel}[2]{%
    \phantomsection
    \def\@currentlabelname{#2}%
    \label{#1}%
}
\makeatother

\usepackage{ifthen}
\numberwithin{equation}{section}

\newtheorem{theorem}{Theorem}[section]
\newtheorem{proposition}[theorem]{Proposition}
\newtheorem{lemma}[theorem]{Lemma}

\crefname{claim}{Claim}{Claims}

\newtheorem{corollary}[theorem]{Corollary}
\newtheorem{conjecture}[theorem]{Conjecture}
\newtheorem*{question*}{Question}

\theoremstyle{definition}

\newtheorem*{definition*}{Definition}

\newtheorem*{fact*}{Fact}

\crefname{fact}{Fact}{Facts}

\theoremstyle{remark}

\DeclareMathOperator{\Aut}{Aut}

\DeclareMathOperator{\parent}{parent}
\DeclareMathOperator{\Bin}{Bin}
\DeclareMathOperator{\inj}{inj}

\newcommand{\eps}{\varepsilon}

\renewcommand{\Pr}{\mathbb P}

\newlength{\entropynamewidth}

\title{Optimal bounds on spanning tree embeddings}
\author{Csongor Beke}\address{(C. Beke) University of Cambridge}\email{cb2138@cam.ac.uk}
\author{Vladimir Bo\v{s}kovi\'c}\address{(V. Bo\v{s}kovi\'c) Alfr\'ed R\'enyi Institute of Mathematics}\email{boskovic@renyi.hu}
\author{Nina Kam\v{c}ev}\address{(N. Kam\v{c}ev) Faculty of Science, University of Zagreb}\email{nina.kamcev@math.hr}
\author{Yiting Wang}\address{(Y. Wang) Institute of Science and Technology Austria}\email{yiting.wang@ist.ac.at}

\date{\today}

\begin{document}

\begin{abstract}
We prove that the number of labelled embeddings of any $n$-vertex tree $T$ into an $n$-vertex graph $G$ of maximum degree $d$ satisfies
\[
    \inj(T,G) \leq (d/e)^n \exp(o_d(1)n).
\]
The bound is sharp up to determining $o_d(1)$, even for paths, and the dependence of the error $\exp(o_d(1)n)$ on $d$ is necessary.
As an immediate corollary, we obtain an optimal anticoncentration bound for the isomorphism class of a uniformly random spanning tree in a connected $d$-regular graph, answering a conjecture of H.~Lee~\cite{Lee}. The proof combines Br\'egman's inequality with entropy methods.
\end{abstract}

\maketitle

\section{Introduction}

\subsection{Counting spanning tree embeddings}
How many spanning trees does a connected graph $G$ contain? Write $\tau(G)$ for this quantity.
For complete graphs, Cayley's formula~\cite{Cayley} gives $\tau(K_n)=n^{n-2}$. 
For a general connected graph $G$, Kirchhoff's matrix-tree theorem~\cite{kirchhoff1847ueber} determines $\tau(G)$ from the spectrum of the Laplacian matrix of $G$.
For connected $d$-regular graphs on $n$ vertices, the work of McKay~\cite{McKay} and Alon~\cite{Alon} gives the asymptotic estimate $\tau(G)=d^n\exp(-o_d(1)n)$. A more explicit error term can be extracted from a result of Kostochka~\cite{kostochka1995number}.

We study a related question: given a tree $T$ on $n$ vertices, how many embeddings of $T$ does a connected graph $G$ on $n$ vertices contain? Recall that an \emph{embedding} is simply an injective map $V(T)\rightarrow V(G)$ that preserves adjacency. Write $\inj(T,G)$ for this number. 
We are interested in upper bounds on $\inj(T,G)$, uniform over $T$, for $d$-regular graphs $G$.

Embedding the vertices successively from an arbitrarily chosen root gives the elementary bound $\inj(T,G)\leq nd^{n-1}$. On the other hand, $\inj(P_n,K_n)=n!=(n/e)^n\exp(O(\log n))$. This suggests saving a factor of approximately $e$ per vertex over the greedy bound. Our main result achieves this saving in any graph of maximum degree at most $d$, up to an error in the exponent that is small when $d$ is large.

\begin{theorem}\label{thm:main}
Let $d\geq 2$ be an integer, let $G$ be an $n$-vertex graph with $\Delta(G)\le d$, and let $T$ be an $n$-vertex tree. Then
\[
    \inj(T,G)\leq\left(\frac de\right)^n
    e^{o_d(1)n},
\]
where the error term $o_d(1)$ can be taken to be $C(\log^2d/d)^{1/3}$ for some universal constant $C>0$.
\end{theorem}
Throughout the introduction, \emph{sharp} or \emph{asymptotically optimal} refers to the leading exponential term. This is also the level of precision of several counting results discussed below. The example $\inj(P_n,K_n)=n!$ shows that even for paths, our bound is optimal in this sense. More generally, the result is optimal throughout the range $d\gg \log n$, as witnessed by the random graph $G(n,p)$. First, note that for $G\sim G(n,p)$ and \textit{any} $n$-vertex tree $T$, we have $\mathbb E\,\inj(T,G)=n!p^{n-1} =\left(\frac{pn}{e}\right)^n e^{o(n)}.$
If $p\gg \log n/n$, then typically $\Delta(G)=(1+o(1))pn$, while a result of Glebov and Krivelevich~\cite{glebov2013hamilton} (see also Janson~\cite{janson1994numbers}) implies that the number of Hamilton paths is $(\Delta(G)/e)^n\exp(o(n))$ with high probability.
Thus, the leading exponential term in \cref{thm:main} is also attained by non-complete host graphs.
Moreover, the dependence of the error term on $d$ is necessary; see \cref{sec:concluding} for details.

Our proof combines Brégman's inequality with entropy methods. Entropy allows us to average the contributions of partial embeddings and their extensions, avoiding worst-case estimates. The main new ingredient is a \emph{sparse fingerprint}, which records a small amount of randomly sampled information about the embedding. Conditioning on this fingerprint gives us sufficiently precise control over the extension counts to obtain the sharp bound, while incurring only a lower-order entropy cost. 

\subsection{Anticoncentration of random spanning trees}
Our result has a direct application to the anticoncentration of random spanning trees.
For a connected graph $G$ on $n$ vertices, define
\[
    \rho(G) \coloneqq \max_{\substack{n\text{-vertex}\\ \text{tree } T}} \Pr(T\cong \mathcal T),
\]
where $\mathcal T$ is chosen uniformly among the spanning trees of $G$. In other words, $\rho(G)$ is the largest probability of any single isomorphism class. We seek upper bounds on this quantity, using the inequality $\rho(G) \leq \max_T \inj(T,G)/\tau(G)$\footnote{In fact, $\rho(G)= \max_T \inj(T,G)/(|\Aut(T)|\cdot \tau(G))$, where $\Aut(T)$ is the automorphism group of $T$. Here we are using the trivial bound $|\Aut(T)|\geq 1$.}.

This problem was first studied by H.~Lee~\cite{Lee}, who proved that $\rho(G)\leq \exp(-\Omega(n))$ for connected almost $d$-regular graphs $G$ on $n$ vertices with $d$ sufficiently large. Bitonti, Michel and Scott~\cite{BitontiMichelScott} proved the bound $\rho(G)\leq n^{-\Omega(d)}$ for connected graphs $G$ of minimum degree $d$ on $n$ vertices with $d$ sufficiently large and $n$ sufficiently large compared to $d$. Both results are sharp up to the constants in the exponents.

Combining \cref{thm:main} with the Alon--McKay estimate for $\tau(G)$ yields the following corollary.
\begin{corollary}\label{cor:anticoncentration}
    For $d\geq 2$, let $G$ be a connected $n$-vertex $d$-regular graph. Then
    \[
        \rho(G)\leq \exp(-(1-o_d(1))n)\,.
    \]
\end{corollary}
This confirms Lee's conjecture~\cite[Conjecture 4.3]{Lee}, also reiterated in~\cite{Wang}.
This result is asymptotically optimal up to determining $o_d(1)$. Indeed, as already observed in~\cite{Lee}, there are connected $d$-regular graphs with many Hamilton paths that match this bound.

\subsection{Related results}
Sharp counting bounds for other spanning structures have been studied extensively. A classical example is Br\'egman's resolution of Minc's conjecture~\cite{Bregman1973}, which bounds the number of perfect matchings in a balanced bipartite graph in terms of the degrees on one side. Br\'egman's inequality is a recurring ingredient in our proof. We recall its statement.
\begin{theorem}[Br\'egman's inequality~\cite{Bregman1973}]\label{thm:bregman}
    Let $G$ be a balanced bipartite graph on vertex set $A\cup B$. Then the number of perfect matchings in $G$ is at most 
    \[
        \prod_{v\in A} (d(v)!)^{1/d(v)}\,.
    \]
\end{theorem}
We use the convention $(0!)^{1/0} = 1$.
The corresponding bound for general (non-bipartite) graphs was shown by Kahn and Lov\'asz (unpublished), as well as Friedland~\cite{Alon-Friedland,friedland2008upper,cutler2011entropy}. 
In Dirac graphs, Cuckler and Kahn~\cite{Cuckler-Kahn-Entropy} established sharp upper bounds on the number of perfect matchings and Hamilton cycles.

\subsubsection{Extremal counting problems.}
The perfect-matching results above belong to a broader family of extremal subgraph-counting problems in graphs with a given density or maximum degree. In the hypergraph setting, Luria~\cite{Luria-n-queens} obtained asymptotic upper bounds on the number of perfect matchings in regular uniform hypergraphs with small codegrees.
Perfect matchings in graphs can also be viewed as $K_2$-factors. Very recently, Lee~\cite{LeeFactors} obtained an asymptotically optimal upper bound on the number of $F$-factors, whenever $F$ is a fixed Hamiltonian graph.
Under a maximum-degree constraint, Chase~\cite{Chase2020} resolved a conjecture of Gan, Loh and Sudakov~\cite{GanLohSudakov2015} by determining the maximum number of triangles in an $n$-vertex graph of maximum degree at most $d$; the extremal graph is a disjoint union of copies of $K_{d+1}$ and one possibly smaller clique.

A more distant comparison is Sidorenko's conjecture~\cite{sidorenko1991inequalities}\footnote{For simplicity, we restrict to dense host graphs.}. For a \textit{fixed} bipartite graph $H$, what is the asymptotic \textit{minimum} of $\inj(H, G)$ over all $n$-vertex graphs of given density $p>0$? Sidorenko proved that when $H$ is a fixed tree, this asymptotic minimum is attained by the random graph $G(n,p)$, and conjectured that the same holds for any bipartite graph $H$. The conjecture remains unresolved despite significant effort -- see~\cite{conlon18-sidorenko,szegedy14} and the references therein.

\subsubsection{Number of distinct spanning trees}
We call two spanning trees \emph{distinct} if they are non-isomorphic, and write $\tau_u(G)$ for the maximum size of a collection of pairwise distinct spanning trees in $G$.
Otter~\cite{Otter} obtained precise asymptotic estimates for $\tau_u(K_n)$, but determining $\tau_u(G)$ is a difficult problem in general, partially because the size of the automorphism group can vary greatly for different spanning trees.
Anticoncentration results such as~\Cref{cor:anticoncentration} give lower bounds on $\tau_u(G)$. Indeed, observe that $\tau_u(G)\geq \rho(G)^{-1}$.
For $(n,d,\lambda)$-graphs satisfying mild spectral assumptions, the fourth author~\cite{Wang} recently proved a lower bound on $\tau_u(G)$ stronger than any obtainable through anticoncentration. Yan and Zhang~\cite{Yan-Zhang} obtained a sharp lower bound on $\tau_u(G)$ for connected graphs of minimum degree $d$ for fixed $d\geq 3$ and sufficiently large $n$.

\subsubsection{Entropy method}
Lastly, we briefly introduce the entropy method, which is the backbone of the proof.
Introduced by Shannon~\cite{Shannon1948} in 1948, entropy measures the average amount of information carried by a random variable, or equivalently, the uncertainty about its value before it is revealed. Its combinatorial applications go back at least to the work of Erd\H{o}s and R\'enyi~\cite{ErdosRenyi1963} on coin-weighing problems. 
Entropy has a natural connection with counting. Proving upper bounds with entropy leads to a fruitful line of research, with results including 
Br\'egman's inequality~\cite{Radhakrishnan1997} (non-bipartite analogue~\cite{cutler2011entropy}), counting hypergraph matchings, $n$-queens configurations~\cite{Luria-n-queens, Simkin-n-queen} and block designs~\cite{Linial-Luria-STS,keevash2018counting}. Entropy has also been used constructively to prove lower bounds, notably in several cases of Sidorenko's conjecture~\cite{szegedy14}.
For comprehensive backgrounds on the entropy method, we refer the readers to the excellent surveys by Galvin~\cite{GalvinEntropy} and by  Samotij~\cite{SamotijEntropy}.

\subsection{Paper organisation}
In~\Cref{sec:overview}, we give an overview of the proof and state some basics of the entropy method. In~\Cref{sec:auxiliary}, we collect various auxiliary results.~\Cref{sec:proof} contains the proof of~\Cref{thm:main}. In~\Cref{sec:concluding}, we prove the necessity of the dependence on $d$ in the theorem and conclude with open problems.

\section{Proof overview}\label{sec:overview}
Recall that greedy embedding gives $\inj(T,G)\le nd^{n-1}$.
Our aim is to save a factor of approximately $e$ per vertex.

We first motivate the proof by considering the case where $T$ is a path. We count directed Hamilton paths in $G$, first fixing their source $s$ and target $t$. There are at most $n^2$ choices for these vertices. Each such path $P$ determines a perfect matching $M_P$ in the auxiliary bipartite graph $\Sigma$ where the vertex classes are copies of $V(G)\setminus\{t\}$ and $V(G)\setminus\{s\}$, and two vertices are adjacent whenever their corresponding vertices are adjacent in $G$: to construct $M_P$, simply match each vertex other than $t$ to its successor on the path $P$. Since $\Sigma$ has maximum
degree at most $d$, Br\'egman's inequality (\Cref{thm:bregman}) gives $\inj(P_n,G)\leq n^2(d!)^{(n-1)/d}\approx (d/e)^n$.

Decomposing a tree with few leaves into paths gives a similar bound, with a small additional error; see \cref{lem:core-embeddings}. For a general tree $T$, we repeatedly remove all leaves, stopping after
$q-1$ rounds or when at most two vertices remain, where $q$ tends
slowly to infinity with $d$. Indexing the removed layers in reverse order gives
\[
    T=L_0\cup L_1\cup\cdots\cup L_h,
\]
where $h<q$ and $L_0$ either has at most $n/q$ leaves or at most $2$ vertices. Each vertex of $L_i$, for $i\geq 1$, has a unique neighbour in the tree $L_0\cup\cdots\cup L_{i-1}$, which we call its \textit{parent}.

To count embeddings of $T$, first fix the image sets $S_0,\ldots,S_h$ of these layers. We embed $L_0$ into $G[S_0]$ and then extend the embedding to $L_1,\ldots,L_h$ in this order. Once the earlier layers have been embedded, extensions to $L_i$ correspond to perfect matchings in a bipartite graph $H_i$ with vertex classes $L_i$ and $S_i$, where $v\in L_i$ is adjacent to $x\in S_i$ if the image of its parent is adjacent to $x$ in $G$. Thus vertices with the same parent have identical neighbourhoods in $H_i$.

Each vertex of $L_i$ has degree at most $d$ in $H_i$, so Br\'egman's inequality bounds the number of extensions by roughly $(d/e)^{|L_i|}$. The corresponding argument for $L_0$ gives the same leading contribution, $(d/e)^{|L_0|}$, with a small error term. Multiplying these bounds gives roughly $(d/e)^n$, as required.

Choosing the image partition may, however, introduce an additional exponential factor. To \emph{pay} for this choice, we apply Brégman’s inequality on the $S_i$ side, retaining the actual degrees. Although these degrees can exceed $d$, their total admits a useful bound. Entropy makes the resulting cancellation explicit.

\subsection{The entropy calculation}
Throughout the paper, all logarithms are natural. For a discrete random variable $X$, its \emph{entropy} is defined by
\[
    H(X)=-\sum_x \mathbb P(X=x)\log\mathbb P(X=x),
\]
with the convention $0\log 0=0$. We use standard facts that the entropy of a random variable supported on a finite set $S$ is at most $\log |S|$, with equality for the uniform distribution, as well as the chain rule, subadditivity and their conditional versions; see~\cite{SamotijEntropy} for background on entropy methods. The additional inequalities needed below are collected in Section~\ref{sec:auxiliary}. 

Let $\Phi$ be a uniformly random embedding of $T$ into $G$, so that $H(\Phi)=\log\inj(T,G)$. Our task is therefore to bound $H(\Phi)$.

We reveal $\Phi$ in the order suggested above: first its image
partition $\mathcal P=(S_0,\ldots,S_h)$, where $S_i=\Phi(L_i)$,
and then the embedding of each successive layer. Write
$\Phi_i=\Phi|_{L_i}$ and $\Phi_{<i} = \Phi|_{L_0\cup\cdots\cup L_{i-1}}$.
The \emph{conditional entropy} $H(\Phi_i\mid\mathcal P,\Phi_{<i})$ measures the average uncertainty in $\Phi_i$ after the partition and earlier embeddings are revealed. Since $\mathcal P$ is determined by $\Phi$, the \emph{chain rule} gives
\begin{equation}\label{eq:overview-chain}
    H(\Phi)=H(\mathcal P)
    +\sum_{i=0}^h H(\Phi_i\mid\mathcal P,\Phi_{<i}),
\end{equation}
where for $i=0$ we condition only on $\mathcal P$. Each conditional entropy is at most the expected logarithm of the number of extensions consistent with the revealed information, which can be bounded using Br\'egman's inequality.

To bound $H(\mathcal P)$, note that the partition is determined by the layer containing each host vertex. Writing $p_i(x)=\mathbb P(x\in S_i)$, the \emph{subadditivity of entropy} gives
\begin{equation}\label{eq:overview-partition}
    H(\mathcal P)\leq
    -\sum_{x\in V(G)}\sum_{i=0}^h p_i(x)\log p_i(x).
\end{equation}

We first explain the cancellation in an idealised setting, replacing
each factor $(r!)^{1/r}$ in Brégman's inequality by $r/e$ and omitting the additional error from embedding the \textit{core} $L_0$. For $i\ge1$, let $R_i(x)$ count the vertices of $L_i$ whose parents have images adjacent to $x$. 
This is defined for every $x\in V(G)$ and equals $d_{H_i}(x)$ when $x\in S_i$.
For the core, set $R_0(x)=|N_G(x)\cap S_0|$. Each vertex of $T$ contributes at most $d$ to the total sum, so
\begin{equation}\label{eq:overview-weight-sum}
    \sum_{x\in V(G)}\sum_{i=0}^h R_i(x)\leq dn.
\end{equation}
Br\'egman's inequality and \eqref{eq:overview-chain} leave us with
the leading term $n\log(d/e)$ (see the derivation of~\eqref{eqn:layerentropy}) and, up to the errors just mentioned,
the additional contribution
\[
    H(\mathcal P)+
    \mathbb E\sum_{x,i}
    \mathbf 1_{\{x\in S_i\}}\log( R_i(x)/d).
\]

Suppose for a moment that the degrees $R_i(x)$ were deterministic. Taking expectations of the indicators and \eqref{eq:overview-partition} gives
$$
H(\mathcal P)+\sum_{x,i}p_i(x)\log\frac{R_i(x)}d
\le \sum_{x,i}p_i(x)\log\frac{R_i(x)}{dp_i(x)}.
$$
\emph{Gibbs' inequality} (\cref{lem:gibbs}), concavity of the log function and \cref{eq:overview-weight-sum} would give
\[
    \sum_{x,i} p_i(x)\log\frac{R_i(x)}{dp_i(x)}\leq \sum_x\log\frac{\sum_iR_i(x)}{d}\leq n\log\frac{\sum_{i,x}R_i(x)}{dn}\leq 0.
\]
Thus, the degree terms would pay for the partition entropy.

In general, the degrees $R_i(x)$ depend on $\Phi$. We could first reveal information determining these degrees and apply the preceding argument conditionally, but revealing this information has an entropy cost. For each layer, the degrees are determined by the numbers of
children whose parents are mapped to each host vertex; for the
core, they are determined by its image set $S_0$. If $C$ is a variable recording all
this information, then conditional on $C$, each $R_i(x)$ is fixed, so the preceding argument applies with $p_i(x)=\mathbb P(x\in S_i\mid C)$. However, the additional cost $H(C)$ may be too large.

\subsection{A sparse fingerprint}\label{sec:fingerprint}
We instead estimate these degrees using a \textit{sparse sample} of the counts. Independently of $\Phi$, mark each vertex of $T$ independently with probability $\theta$. For $i\ge1$, let $Z_i(x)$ count the marked vertices of $L_i$ whose parents map to $x$, and for $i=0$, let $Z_0(x)$ indicate whether $x$ is the image of a marked core vertex. Write $Z=(Z_i(x))_{i,x}$. We view $Z$ as a randomised \emph{fingerprint} of the embedding:
it retains useful information about the degrees at a small
entropy cost. The use of a sparse random sample to record
information about a much larger object is reminiscent of
the container constructions of Saxton and
Thomason~\cite{SaxtonThomasonSimple}.

The fingerprint $Z$ has three properties that are crucial for the argument. Denote the \textit{subsampled variant} of $R_i(x)$ by  $N_i(x)$ so that, conditional on $\Phi$, $N_i(x) \sim \Bin (R_i(x), \theta)$. Firstly, the expected logarithm of $N_i(x)$ can be directly related to $\log R_i(x)$ via~\cref{cla:binomial}, making $N_i(x)$ a useful estimate for $R_i(x)$ in the context of entropy. Secondly, $Z$ determines every $N_i(x)$. Thirdly (see below), the entropy $H(Z)$ is of order $o_d(n)$.

We bound the cost of the partition by first revealing the fingerprint:
\[
    H(\mathcal P)\leq H(Z)+H(\mathcal P\mid Z).
\]
Conditional on $Z$, the degree estimates $N_i(x)$ are fixed, so we can repeat
the preceding cancellation using the conditional probabilities
$\mathbb P(x\in S_i\mid Z)$.
The sampled degrees have total expectation at most $\theta dn$,
so the final simplification using concavity leaves essentially a lower-order correction term introduced
by the binomial estimate.
Thus the degree terms cancel with $H(\mathcal P\mid Z)$ up to this
correction, leaving the information cost $H(Z)$.

The fingerprint $Z$ has at most $nq$ nonnegative integer coordinates, with total mean $\theta n$.  \cref{lem:maxentropy} is a mean-dependent upper bound on $H(Z_i(x))$, which together with subadditivity and concavity gives
\[H(Z)\leq \theta n\log(4q/\theta).\]
Choosing $q$ and $\theta$ appropriately makes this cost, the sampling correction and the core embedding error all $o_d(1)n$.

\section{Auxiliary lemmas}
    \label{sec:auxiliary}
We prove the auxiliary results used in the overview, beginning with the bound for trees with few leaves.

\begin{lemma}\label{lem:core-embeddings}
For an integer $d\ge2$, let $L_0$ be a tree with $s+1$ leaves,
where $s\ge1$, and let $\Gamma$ be a graph on the same number
of vertices with $\Delta(\Gamma)\le d$. Then we have
\[
 \log\inj(L_0,\Gamma)
 \le |L_0|\log(d/e)
 +\sum_{x\in V(\Gamma)}\log\frac{r(x)+\theta^{-1}}{d}
 +\varepsilon_0 s,
\]
where $r(x)=d_\Gamma(x)$, $\theta\in(0,1)$ is any real number
satisfying $3\log(d+1)+e\le\theta^{-1}$, and $\varepsilon_0=
 \log(2e^3d|L_0|^2\theta/s^2)$ is an error term.
\end{lemma}

\begin{proof}
Write $m=|L_0|$.  We may assume that $\inj(L_0,\Gamma)>0$. Root $L_0$ at a leaf, orient its edges away
from the root, and retain one fixed edge to a child at every vertex that has children. Each vertex now has indegree and outdegree at most one, so the retained edges form directed paths,
possibly consisting of single vertices. Their terminal vertices are precisely the $s$ non-root leaves, so there are $s$ paths. They contain $m-s$ edges in total, and hence $s-1$ edges were deleted.

Every embedding $\Phi$ induces an injection
$f:V(\Gamma)\setminus U\to V(\Gamma)$, defined by
$f(\Phi(v))=\Phi(w)$ on each retained edge $v\to w$,
satisfying $f(x)\sim_\Gamma x$, where $U$ is the set of images
of the $s$ non-root leaves. For each such $f$, there are at
most $sd^{s-1}$ corresponding embeddings: the root image has at most $s$ choices outside the image of $f$, and traversing the rooted tree, each retained edge determines the next image through $f$, while each deleted edge gives at most $d$ choices.

Fix $U\subseteq V(\Gamma)$ of size $s$. Form a bipartite
graph whose left class consists of $V(\Gamma)\setminus U$ together with
$s$ additional labelled vertices, and whose right class is a copy of $V(\Gamma)$. Ordinary vertices inherit their edges from
$\Gamma$, while each additional vertex is adjacent to all of
$V(\Gamma)$. Every such $f$ extends to exactly $s!$ perfect
matchings. Thus Br\'egman's inequality bounds the number
of such maps $f$ by
\[
 \frac{(m!)^{s/m}}{s!}
 \prod_{x\in V(\Gamma)\setminus U}(r(x)!)^{1/r(x)}.
\]
By Stirling's estimate,
$(r!)^{1/r}\le(r+3\log(r+1))/e
\le(r+\theta^{-1})/e$ for $1\le r\le d$.
For $x\in U$ we have
$(r(x)+\theta^{-1})/e\ge 1/(e\theta)$,
so summing over $U$, we obtain
\[
 \inj(L_0,\Gamma)
 \le sd^{s-1}\binom ms\frac{(m!)^{s/m}}{s!}
 (e\theta)^s
 \prod_{x\in V(\Gamma)}\frac{r(x)+\theta^{-1}}{e}.
\]
Using $\binom ms\le(em/s)^s$, $s!\ge(s/e)^s$,
$(m!)^{s/m}\le m^s$, and $s/d\le2^s$, it follows that
\[
 \inj(L_0,\Gamma)
 \le \left(\frac{2e^3dm^2\theta}{s^2}\right)^s
 \prod_{x\in V(\Gamma)}\frac{r(x)+\theta^{-1}}{e}.
\]
Taking logarithms completes the proof.
\end{proof}

We next record the auxiliary inequalities used in the proof. They are standard, but we include the short proofs for completeness. We use the following form of Gibbs' inequality.

\begin{lemma}[Gibbs' inequality]\label{lem:gibbs}
Let $(p_i)$ and $(q_i)$ be nonnegative sequences with
$P=\sum_i p_i\in(0,\infty)$ and
$Q=\sum_i q_i\in(0,\infty)$. Then
\[
    \sum_i p_i\log\frac{q_i}{p_i}
    \le P\log\frac{Q}{P}.
\]
Terms with $p_i=0$ are zero; if $p_i>0=q_i$ for some $i$,
the left-hand side is $-\infty$.
\end{lemma}

\begin{proof}
Apply Jensen's inequality to the $\log$ function with weights $p_i/P$
over the indices with $p_i>0$.
\end{proof}

The following estimate relates $\log R$ to the expected logarithm of a binomial subsample.
\begin{lemma}\label{cla:binomial}
Let $N\sim\Bin(R,\theta)$, where $R$ is a nonnegative integer
and $0<\theta<1$. Then
\[
    \log\left(R+\theta^{-1}\right)
    \le\mathbb E\log\frac{N+2}{\theta}.
\]
\end{lemma}

\begin{proof}
Using $1-u\le e^{-u}$, we have
\[
    \mathbb E\frac1{N+2}
    =\int_0^1 s(1-\theta+\theta s)^R\,ds
    \le\int_0^1 e^{-(1+\theta R)(1-s)}\,ds
    \le\frac1{1+\theta R}.
\]
Now apply Jensen's inequality for $-\log$
and subtract $\log\theta$.
\end{proof}

The next lemma states that among nonnegative integer-valued random variables of a given mean, the geometric distribution maximises entropy.

\begin{lemma}\label{lem:maxentropy}
Let $X$ be a nonnegative integer-valued random variable
with mean $t<\infty$.
Then $H(X)\le\psi(t)$, where
$\psi(t)=(1+t)\log(1+t)-t\log t$ for $t\ge0$.
The function $\psi$ is concave, $\psi(0)=0$, and
$\psi(t)\le t\log(4/t)$ for $0<t\le1$.
\end{lemma}

\begin{proof}
For $t>0$, apply Gibbs' inequality with $p_i=\Pr(X=i)$
and $q_i=t^i/(1+t)^{i+1}$ for $i\ge0$ to obtain
$H(X)\le-\sum_i p_i\log q_i=\psi(t)$.
The case $t=0$ is immediate.
Concavity of $\psi$ follows from $\psi''<0$.
For $t \in (0,1]$, convexity of $(1+t)\log (1+t)$ gives
$(1+t)\log(1+t)\le2t\log2$, and hence
$\psi(t)\le t\log(4/t)$.
\end{proof}

\section{Proof of the main theorem}\label{sec:proof}
\begin{proof}[Proof of \cref{thm:main}] We may assume $d$ is sufficiently large.
Indeed, for bounded $d$, the result follows form the trivial greedy bound $\inj(T,G)\leq nd^{n-1}$ after increasing $C$.
Choose an integer $q\geq 2$ to be specified later, satisfying $3\log(d^2+1)+e\le q$, and set $\theta = 1/q$. We may assume that $\inj(T,G)>0$, and hence $\Delta(T)\le d$.

We iteratively remove all leaves of our tree, until we have either reached a tree with at most $2$ vertices, or have removed $q-1$ layers of leaves. Call the tree at the end of the process $L_0$, and index the removed layers backwards as $L_1,\dots,L_h$ where $h\leq q-1$. For each $v\in L_i$, $i>0$ we write $\parent(v)$ for the unique neighbour of $v$ in $L_0\cup\dots \cup L_{i-1}$. Every vertex of $L_i$ for $0<i<h$ (or leaf of $L_0$) has a child in the following layer, so
$$\text{(number of leaves of $L_0$)}\leq |L_1|\leq \dots \leq |L_h|.$$
In particular, either $|L_0|\leq 2$, or $h=q-1$ and $L_0$ has at most $n/q$ leaves.

Let $\Phi$ be a uniformly random embedding of $T$ into $G$. Then $H(\Phi)=\log(\inj(T,G))$. For $0\leq i\leq h$, let $S_i:=\{\Phi(v):v\in L_i\}$ be the set of images of the vertices in $L_i$ and write $\mathcal{P}=(S_0,\dots , S_h)$ for the partition of the vertex set of $G$. We define $\Phi_i\coloneqq \Phi|_{L_i}$ and $\Phi_{<i}\coloneqq \Phi|_{L_0\cup\cdots\cup L_{i-1}}
$.
By the chain rule for entropy, we have
$$H(\Phi)=H(\mathcal{P})+\sum_{i=0}^{h}H(\Phi_i\mid\mathcal{P},\Phi_{<i}).$$
Conditional on $\mathcal P$ and $\Phi_{<i}$, the number of possible
values of $\Phi_i$ is at most the number of extensions to $L_i$
with image set $S_i$. 

For each $x\in V(G)$, write\footnote{Note that for $i \geq 1$, $C_i(x)$ is determined by $\Phi_{<i}$, and that $C_i(x) =0$ if $x$ is assigned to a part $S_j$ with $j \geq i$.} 
\[
C_i(x)=
\begin{cases}
\bigl|\{v\in L_i:\Phi_{<i}(\parent(v))=x\}\bigr|, & i>0,\\
\mathbf{1}_{\{x\in S_0\}}, & i=0,
\end{cases}
\qquad\text{and}\qquad
R_i(x)=\sum_{y\sim x}C_i(y).
\]
Note that for $i>0$, $C_i(x)$ is the number of children to be embedded in layer $i$ for a parent mapped to $x$, and since $\Delta(T),\Delta(G)\leq d$ we have $R_i(x)\leq d^2$. We also have
$$\sum_xC_i(x)=|L_i|\text{ for all }i\geq 0,\quad \text{and}\quad \sum_{i,x}R_i(x)\leq dn\,,$$
recalling the latter inequality from~\eqref{eq:overview-weight-sum}.
For $i>0$, we define an auxiliary bipartite graph $H_i$ with vertex classes $L_i$ and $S_i$, so that
$$vx\in E(H_i) \quad \Longleftrightarrow \quad \Phi_{<i}(\parent(v))x\in E(G).$$
Note that the extensions we are looking for are in bijection with the perfect matchings of $H_i$ and for each $x\in S_i$ we have $d_{H_i}(x)=R_i(x)$. Therefore, by Brégman's inequality, the number of extensions is bounded above by
$$\prod_{x\in S_i}(R_i(x)!)^{1/R_i(x)}.$$
Using a standard Stirling's estimate and our choice of $\theta$, for any $0<r\leq d^2$ integer, we have 
$$(r!)^{1/r}\leq \frac{r+3\log(r+1)}{e}\leq \frac{r+3\log(d^2+1)}{e}\leq \frac{d}{e}\cdot \frac{r+\theta^{-1}}{d}.$$
Consequently, for every $x \in S_i$,
$
     \log \left(R_i(x)! \right)^{1/R_i(x)} \leq \log (d/e) + \log ((R_i(x)+\theta^{-1})/d).
$
The first term is the desired main contribution: after summing over all vertices, it contributes $n \log (d/e)$. We retain the second term, as it is used to compensate for the entropy of the image partition. 

Therefore, by considering the support of each conditional random variable, we obtain
\begin{equation}\label{eqn:layerentropy}
    H(\Phi_i\mid\mathcal{P},\Phi_{<i})\leq |L_i|\log\left(d/e\right)+
    \mathbb{E}\sum_{x\in S_i}\log\frac{R_i(x)+\theta^{-1}}{d}.
\end{equation}
Set $\varepsilon=q^{-1}\log(2e^3\theta dq^2)$.
If $|L_0|>2$, let $s+1$ be the number of leaves of $L_0$.
Applying \cref{lem:core-embeddings} with $\Gamma=G[S_0]$,
the error term is at most
$$
s\log\frac{2e^3d\theta |L_0|^2}{s^2}
\le \frac nq\log(2e^3\theta dq^2)
=\varepsilon n,
$$
where we used $|L_0|\le n$, $s\le n/q$ and $\theta dq\ge1$.
If $|L_0|\leq 2$, the required bound holds without an error term,
since $\theta^{-1}\ge2e$ and $\inj(L_0,G[S_0])\leq 2^{|L_0|}\leq \prod_{x\in S_0}(R_0(x)+\theta^{-1})/e.$

Thus \eqref{eqn:layerentropy} holds for $i=0$ with an additional error term
$\varepsilon n$, and the entropy decomposition gives
\begin{equation}\label{eq:startingpoint}
    H(\Phi)\leq n\log\left(d/e\right)+\eps n + H(\mathcal{P})+
    \mathbb E\sum_{i,x}\mathbf{1}_{x\in S_i}\log\frac{R_i(x)+\theta^{-1}}{d}.
\end{equation}
To analyse the last two summands, mark each vertex of $T$ independently with probability $\theta$ (independent of $\Phi$), and let $M$ be the set of marked vertices. For each $x\in V(G)$ define
\[
Z_i(x)=
\begin{cases}
\bigl|\{v\in L_i \cap M:\Phi_{<i}(\parent(v))= x\}\bigr|, & i>0,\\
\bigl|\{v\in L_0 \cap M:\Phi(v)=x\}\bigr|, & i=0,
\end{cases}
\quad\text{and}\quad
N_i(x)=\sum_{y\sim x}Z_i(y).
\]
Conditional on $\Phi$, $N_i(x)$ counts marked vertices in a fixed
set of size $R_i(x)$, so $N_i(x)\sim\Bin(R_i(x),\theta)$.
Thus $N_i(x)/\theta$ is an unbiased estimate of $R_i(x)$, with the
precise logarithmic bound given by \cref{cla:binomial}.

Writing $Z=(Z_i(x))_{i,x}$ and $p_i(x)=\mathbb{P}(x\in S_i\mid Z)$, for each $x\in V(G)$ we have $\sum_{i=0}^hp_i(x)=1$. 
By the subadditivity of conditional entropy, we obtain
\begin{equation}\label{eq:P-Zentropy}
    H(\mathcal{P})\leq H(\mathcal{P},Z)=H(Z)+H(\mathcal{P}\mid Z)\leq H(Z)-\mathbb{E}\sum_{i,x}p_i(x)\log p_i(x).
\end{equation} 

To make clear which randomness each expectation averages over,
we write $\mathbb E_\Phi$ for expectation over the embedding $\Phi$,
$\mathbb E_M$ for expectation over the markings with $\Phi$ fixed,
and $\mathbb E$ for expectation over both.

Apply~\cref{cla:binomial} conditionally on $\Phi$  for each $i$ and $x$, 
$ \log((R_i(x)+\theta^{-1})/d) \leq
\mathbb E_M\left[ \log((N_i(x)+2)/(\theta d))\right]$. 
Now average over $\Phi$ and note that 
$\mathbb E \left[ \mathbf{1}_{\{x\in S_i\}} \log (N_i(x)+2) \mid Z \right] = p_i(x) \log (N_i(x)+2)$ since $N_i(x)$ is determined by $Z$. 
Thus, for any fixed $x\in V(G)$, we have
\begin{equation}\label{eq:apply33}
    \mathbb E_\Phi\sum_{i}\mathbf{1}_{x\in S_i}\log\frac{R_i(x)+\theta^{-1}}{d}\leq \mathbb{E}_{\Phi}\sum_i
\mathbf{1}_{\{x\in S_i\}}\,
\mathbb{E}_{M}\!\left[
\log\frac{N_i(x)+2}{\theta d}
\right] = \mathbb E \sum_i p_i(x) \log\frac{N_i(x)+2}{\theta d}\,.
\end{equation}
We now combine this estimate with the entropy terms $-p_i(x)\log p_i(x)$ from~\eqref{eq:P-Zentropy}. For each $x\in V(G)$, we have $\sum_i p_i(x)=1$, so Gibbs' inequality gives
\begin{equation}\label{eq:Gibbs-application}
\sum_i p_i(x)\log\frac{N_i(x)+2}{\theta d\,p_i(x)} \leq \log\frac{\sum_i(N_i(x)+2)}{\theta d}.
\end{equation}

Taking expectations in \eqref{eq:Gibbs-application}, summing over $x\in V(G)$ and combining ~\eqref{eq:startingpoint},  ~\eqref{eq:P-Zentropy} and ~\eqref{eq:apply33}, we obtain
\[
H(\Phi)\le
n\log(d/e)+\eps n+H(Z)
+\sum_x\mathbb E\log\frac{\sum_i(N_i(x)+2)}{\theta d}.
\]
Since $\mathbb E\sum_{i,x}N_i(x) =\theta\,\mathbb E\sum_{i,x}R_i(x) \le\theta dn$ and $h+1\le q$, applying Jensen's inequality over the vertices and the randomness, gives
$$H(\Phi)\leq n\log(d/e)+\eps n+H(Z)+n\log\left(1+\frac{2q}{\theta d}\right).$$

Each marked vertex contributes to exactly one coordinate of $Z$, so $Z$ may be viewed as an $nq$-coordinate integer sequence with total mean $\theta n$ (after padding with zeros). By subadditivity, \cref{lem:maxentropy} and concavity of $\psi$, we get
$$H(Z)\leq \sum_{i,x}H(Z_i(x))\leq \sum_{i,x}\psi(\mathbb{E}Z_i(x))\leq nq\psi(\theta/q).$$
Using the second part of~\cref{lem:maxentropy} and $\theta/q<1$, we conclude that $H(Z) \leq \theta n \log (4q / \theta)$.

Combining these bounds gives 
$$H(\Phi)\leq n\log\left(\frac de\right)+\theta n\log \frac{4q}{\theta}+n\log\left(1+\frac{2q}{\theta d}\right)+\frac{n}{q}\log(2e^3\theta d q^2).$$
Take $q=\lceil(d\log d)^{1/3}\rceil$ and recall that $\theta=1/q$.
For sufficiently large $d$, these indeed satisfy $3\log(d^2+1)+e\leq q$ and give a total error of $O(n(\log^2 d/d)^{1/3})$, finishing the proof.
\end{proof}

\section{Concluding remarks}\label{sec:concluding}
We prove that the dependence on $d$ in the error term of~\Cref{thm:main} is necessary.
\begin{proposition}\label{prop:example}
    For any given $d\geq 2$, there exists $c_d>0$ such that for infinitely many $n\in \mathbb N$, there are $n$-vertex graphs $G$ with maximum degree at most $d$ and a tree $T$, such that $\inj(T,G)> (d/e)^n e^{c_d n}$.
\end{proposition}
\begin{proof}
Set $r=d-1$ and let $G=T$ be the complete rooted $r$-ary tree of height $h$, with $n$ vertices. Then $\Delta(G)\le d$, and permuting the identical branches at each of its $(n-1)/r$ internal vertices gives
\[
    \inj(T,G)=|\Aut(T)|\ge (r!)^{(n-1)/r}.
\]
By Stirling's lower bound, $b_d:=(r!)^{1/r}>d/e$ for $r\ge2$, and this also holds directly for $r=1$. Thus $\delta_d:=\log(eb_d/d)>0$, and
\[
    \inj(T,G)\ge b_d^{n-1}
    =\left(\frac de\right)^n\exp\!\left(\delta_dn-\log b_d\right)
    >\left(\frac de\right)^n e^{\delta_dn/2}
\]
for all sufficiently large $h$. Taking $c_d=\delta_d/2$ proves the claim.
\end{proof}
In fact, the same argument applies even when restricting $G$ to be $d$-regular. Indeed, join the roots of two copies of $T$ to form $T'$, then add an $r$-regular graph on its $2r^h$ leaves to obtain $G'$; since $T'$ has $2n$ vertices and $\inj(T',G')\ge(r!)^{2(n-1)/r}$, the same argument applies.

It would be interesting to determine the smallest possible function
$f(d)$ for which the error term in \cref{thm:main} can be taken to be $nf(d)$. 
We make the following conjecture:
\begin{conjecture}
    There exists an absolute constant $C>0$, such that for any $d\geq 2$ and any $n$-vertex tree $T$ and any $n$-vertex graph $G$ with $\Delta(G)\leq d$, we have 
    \[
        \inj(T,G)\leq \left(\frac{d}{e}\right)^n \exp\left(Cn\frac{\log d}{d}\right)\,.
    \]
\end{conjecture}
There is some evidence for this conjecture. For instance, it approximately matches the lower bound in~\Cref{prop:example}. Also,~\Cref{lem:core-embeddings} implies it is true for $n$-vertex trees with at most $O(n/d)$ leaves.

It would also be interesting to determine when a matching lower bound to \cref{thm:main} holds. For $G \sim G(n,p)$ and any $n$-vertex tree $T$, $\mathbb E \, \inj(T,G) = \left( \frac{pn}{e} \right)^n e^{o(n)}$, which suggests a natural direction -- for which classes of random or pseudorandom graphs $G$ one (typically) has
$$ \inj(T,G) \geq \left(\frac de\right)^n e^{-o(n)}\,$$
for all trees $T$ in a suitable family? Such a statement cannot hold for arbitrary spanning trees, e.g.~$T= K_{1,n-1}$, so bounded-degree trees are a natural first case to consider. When $T$ is a path, bounds of this form for (pseudo)random graphs follow from Hamilton-cycle counting results of Glebov and Krivelevich~\cite{glebov2013hamilton,Krivelevich2012Pseudo}. 

When $T$ is a bounded-degree tree, the existence problem is already classical and  very difficult~\cite{montgomery2019,hyde2025}, but the counting problem may be amenable e.g.~for dense host graphs. 

\section*{Acknowledgements}
We would like to thank the organisers of the \href{https://www.renyi.hu/~emlektab/main_news.html}{20th Eml\'ekt\'abla workshop}, where this project was initiated. We also thank Julian Sahasrabudhe for helpful comments on the manuscript.\\
V.B.~is supported by the European Research Council (ERC), via grant agreement ``ERMiD'' No.\ 101054936, and by the Hungarian National Research, Development and Innovation Office, Advanced grant 153378. \\
N.K.~was supported in part by the Croatian Science Foundation under the project number HRZZ-IP-2026-2389 (TEPiC).\\
Y.W.~is supported by the European Research Council (ERC), via grant agreement ``RANDSTRUCT'' No.\ 101076777.

\section*{Declaration on the usage of generative AI}
All ideas, arguments and calculations in the paper come from the authors. During the preparation of this manuscript, the authors used ChatGPT 5.6 Sol as a tool for proofreading and local editing only.


\begin{thebibliography}{10}

\bibitem{Alon}
Noga Alon.
\newblock The number of spanning trees in regular graphs.
\newblock {\em Random Structures \& Algorithms}, 1(2):175--181, 1990.

\bibitem{Alon-Friedland}
Noga Alon and Shmuel Friedland.
\newblock The maximum number of perfect matchings in graphs with a given degree sequence.
\newblock {\em Electron. J. Combin.}, 15(1):Note 13, 2, 2008.

\bibitem{BitontiMichelScott}
Veronica Bitonti, Lukas Michel, and Alex Scott.
\newblock Anticoncentration of random spanning trees in graphs with large minimum degree, 2026.
\newblock arXiv:2603.17630.

\bibitem{Bregman1973}
L.~M. Br{\`e}gman.
\newblock Some properties of nonnegative matrices and their permanents.
\newblock {\em Soviet Math. Dokl.}, 14:945--949, 1973.

\bibitem{Cayley}
Arthur Cayley.
\newblock A theorem on trees.
\newblock {\em Quart. J. Pure Appl. Math.}, 23:376--378, 1889.

\bibitem{Chase2020}
Zachary Chase.
\newblock The maximum number of triangles in a graph of given maximum degree.
\newblock {\em Advances in Combinatorics}, 2020.

\bibitem{conlon18-sidorenko}
David Conlon, Jeong~Han Kim, Choongbum Lee, and Joonkyung Lee.
\newblock Some advances on {S}idorenko's conjecture.
\newblock {\em J.\ Lond.\ Math.\ Soc.\ (2)}, 98(3):593--608, 2018.

\bibitem{Cuckler-Kahn-Entropy}
Bill Cuckler and Jeff Kahn.
\newblock Entropy bounds for perfect matchings and {H}amiltonian cycles.
\newblock {\em Combinatorica}, 29(3):327--335, 2009.

\bibitem{cutler2011entropy}
Jonathan Cutler and A.~J. Radcliffe.
\newblock An entropy proof of the {K}ahn-{L}ov\'{a}sz theorem.
\newblock {\em Electron. J. Combin.}, 18(1):Paper 10, 9, 2011.

\bibitem{ErdosRenyi1963}
Paul Erd{\H{o}}s and Alfr{\'e}d R{\'e}nyi.
\newblock On two problems of information theory.
\newblock {\em Magyar Tud. Akad. Mat. Kutat{\'o} Int. K{\"o}zl.}, 8(1--2):229--243, 1963.

\bibitem{friedland2008upper}
Shmuel Friedland.
\newblock An upper bound for the number of perfect matchings in graphs.
\newblock {\em arXiv preprint arXiv:0803.0864}, 2008.

\bibitem{GalvinEntropy}
David Galvin.
\newblock Three tutorial lectures on entropy and counting, 2014.
\newblock arXiv:1406.7872.

\bibitem{GanLohSudakov2015}
Wenying Gan, Po-Shen Loh, and Benny Sudakov.
\newblock Maximizing the number of independent sets of a fixed size.
\newblock {\em Combin. Probab. Comput.}, 24(3):521--527, 2015.

\bibitem{glebov2013hamilton}
Roman Glebov and Michael Krivelevich.
\newblock On the number of hamilton cycles in sparse random graphs.
\newblock {\em SIAM Journal on Discrete Mathematics}, 27(1):27--42, 2013.

\bibitem{hyde2025}
Joseph Hyde, Natasha Morrison, Alp M\"{u}yesser, and Mat\'{\i}as Pavez-Sign\'{e}.
\newblock Spanning trees in pseudorandom graphs via sorting networks.
\newblock {\em Proc. Amer. Math. Soc.}, 153(6):2353--2367, 2025.

\bibitem{janson1994numbers}
Svante Janson.
\newblock The numbers of spanning trees, hamilton cycles and perfect matchings in a random graph.
\newblock {\em Combinatorics, Probability and Computing}, 3(1):97--126, 1994.

\bibitem{keevash2018counting}
Peter Keevash.
\newblock Counting designs.
\newblock {\em Journal of the European Mathematical Society}, 20(4):903--927, 2018.

\bibitem{kirchhoff1847ueber}
Gustav Kirchhoff.
\newblock Ueber die aufl{\"o}sung der gleichungen, auf welche man bei der untersuchung der linearen vertheilung galvanischer str{\"o}me gef{\"u}hrt wird.
\newblock {\em Annalen der Physik}, 148(12):497--508, 1847.

\bibitem{kostochka1995number}
Alexandr~V. Kostochka.
\newblock The number of spanning trees in graphs with a given degree sequence.
\newblock {\em Random Structures \& Algorithms}, 6(2-3):269--274, 1995.

\bibitem{Krivelevich2012Pseudo}
Michael Krivelevich.
\newblock On the number of {H}amilton cycles in pseudo-random graphs.
\newblock {\em Electron. J. Combin.}, 19(1):Paper 25, 14, 2012.

\bibitem{Lee}
Hyunwoo Lee.
\newblock Anticoncentration of random spanning trees in almost regular graphs, 2026.
\newblock arXiv:2601.07740.

\bibitem{LeeFactors}
Hyunwoo Lee.
\newblock Kahn--{Lov{\'a}sz}-type inequalities for graph factors, 2026.
\newblock arXiv:2608.20303.

\bibitem{Linial-Luria-STS}
Nathan Linial and Zur Luria.
\newblock An upper bound on the number of {S}teiner triple systems.
\newblock {\em Random Structures Algorithms}, 43(4):399--406, 2013.

\bibitem{Luria-n-queens}
Zur Luria.
\newblock New bounds on the number of n-queens configurations, 2017.

\bibitem{McKay}
Brendan~D. McKay.
\newblock Spanning trees in regular graphs.
\newblock {\em European J. Combin.}, 4(2):149--160, June 1983.

\bibitem{montgomery2019}
Richard Montgomery.
\newblock Spanning trees in random graphs.
\newblock {\em Adv. Math.}, 356:106793, 92, 2019.

\bibitem{Otter}
Richard Otter.
\newblock The number of trees.
\newblock {\em Ann. of Math. (2)}, 49(3):583--599, 1948.

\bibitem{Radhakrishnan1997}
Jaikumar Radhakrishnan.
\newblock An entropy proof of {Bregman}'s theorem.
\newblock {\em J. Combin. Theory Ser. A}, 77(1):161--164, 1997.

\bibitem{SamotijEntropy}
Wojciech Samotij.
\newblock Entropy methods in combinatorics, 2026.
\newblock arXiv:2607.24414.

\bibitem{SaxtonThomasonSimple}
David Saxton and Andrew Thomason.
\newblock Simple containers for simple hypergraphs.
\newblock {\em Combin. Probab. Comput.}, 25(3):448--459, 2016.

\bibitem{Shannon1948}
Claude~Elwood Shannon.
\newblock A mathematical theory of communication.
\newblock {\em The Bell system technical journal}, 27(3):379--423, 1948.

\bibitem{sidorenko1991inequalities}
Alexander~F Sidorenko.
\newblock Inequalities for functionals generated by bipartite graphs.
\newblock {\em Diskretnaya Matematika}, 3(3):50--65, 1991.

\bibitem{Simkin-n-queen}
Michael Simkin.
\newblock The number of {$n$}-queens configurations.
\newblock {\em Adv. Math.}, 427:Paper No. 109127, 83, 2023.

\bibitem{szegedy14}
B~Szegedy.
\newblock An information theoretic approach to {S}idorenko’s conjecture.
\newblock {\em arXiv:1406.6738}, 2014.

\bibitem{Wang}
Yiting Wang.
\newblock On the number of distinct spanning trees in pseudorandom graphs, 2026.
\newblock arXiv:2605.12742.

\bibitem{Yan-Zhang}
Zhifei Yan and Lu-Ming Zhang.
\newblock An extremal theorem for non-isomorphic spanning trees, 2026.
\newblock arXiv:2609.30201.

\end{thebibliography}
\end{document}